\documentclass[12pt]{article} 

\usepackage[margin=1in]{geometry}

\usepackage{amssymb,amsfonts,amsmath,latexsym,amsthm}
\usepackage[usenames,dvipsnames]{color}
\usepackage[]{graphicx}
\usepackage[space]{grffile}
\usepackage{mathrsfs}   
\usepackage[skip=0pt]{caption}
\usepackage{subcaption}
\usepackage{verbatim}
\usepackage{url}
\usepackage{bm}
\usepackage{dsfont}
\usepackage{extarrows}
\usepackage{multirow}
\usepackage{tikz}
\usetikzlibrary{fit}					

\usepackage[breaklinks=true]{hyperref}
\usepackage{breakcites}

\usepackage[authoryear,round]{natbib}

\theoremstyle{plain}
\newtheorem{theorem}{Theorem}[section]
\newtheorem*{theorem*}{Theorem} 

\newtheorem{lemma}[theorem]{Lemma}

\theoremstyle{definition}

\theoremstyle{remark}

\allowdisplaybreaks 

\DeclareMathOperator*{\DP}{DP}
\renewcommand{\v}{\bm v}
\newcommand{\z}{\bm z}
\newcommand{\x}{\bm x}
\newcommand{\y}{\bm y}
\newcommand{\q}{\bm q}
\newcommand{\C}{\bm C}

\DeclareMathOperator*{\Beta}{Beta}

\newcommand{\N}{\mathbb{N}}

\renewcommand{\Pr}{\mathbb{P}}
\newcommand{\I}{\mathds{1}}

\newcommand{\iid}{\stackrel{\mathrm{iid}}{\sim}}

\def\app#1#2{%
  \mathrel{%
    \setbox0=\hbox{$#1\sim$}%
    \setbox2=\hbox{%
      \rlap{\hbox{$#1\propto$}}%
      \lower1.3\ht0\box0%
    }%
    \raise0.25\ht2\box2%
  }%
}

\title{An elementary derivation of the Chinese restaurant process from Sethuraman's stick-breaking process}
\author{Jeffrey W. Miller\\
Harvard University, Department of Biostatistics}

\begin{document}
\maketitle
\begin{abstract}
The Chinese restaurant process (CRP) and the stick-breaking process are the two most commonly used representations of the Dirichlet process. However, the usual proof of the connection between them is indirect, relying on abstract properties of the Dirichlet process that are difficult for nonexperts to verify. This short note provides a direct proof that the stick-breaking process leads to the CRP, without using any measure theory.
We also discuss how the stick-breaking representation arises naturally from the CRP.
\end{abstract}

\section{Introduction}
\citet{Sethuraman_1994} showed that the Dirichlet process has the following stick-breaking representation: if
$\v_1,\v_2,\ldots\iid\Beta(1,\alpha)$, $\bm\pi_k = \v_k\prod_{i = 1}^{k-1} (1 -\v_i)$ for $k = 1,2,\ldots$, and $\bm\theta_1,\bm\theta_2,\ldots\iid H$, then the random discrete measure
\begin{align}\label{equation:stick}
\bm P = \sum_{k = 1}^\infty \bm\pi_k \delta_{\bm\theta_k}
\end{align}
is distributed according to the Dirichlet process $\DP(\alpha,H)$ with concentration parameter $\alpha$ and base distribution $H$. 
This representation has been instrumental in the development of many nonparametric models \citep{MacEachern_1999,MacEachern_2000,Hjort_2000,Ishwaran_2000,Ishwaran_2001b,Griffin_2006,Dunson_2008,Chung_2009,Rodriguez_2011,Broderick_2012}, has facilitated the understanding of these models \citep{Favaro_2012,Teh_2007,Thibaux_2007,Paisley_2010}, and is central to various inference algorithms \citep{Ishwaran_2001b,Blei_2006,Papaspiliopoulos_2008,Walker_2007,Kalli_2011}.

It is well-known that, as shown by \citet{Antoniak_1974}, the Dirichlet process induces a distribution on partitions as follows:
if $\bm P\sim\DP(\alpha,H)$ where $H$ is nonatomic (i.e., $H(\{\theta\})=0$ for any $\theta$), $\x_1,\ldots,\x_n|\bm P\,\iid\, \bm P$, and $\C$ is the partition of $\{1,\ldots,n\}$ induced by
$\x_1,\ldots,\x_n$, then
\begin{align}\label{equation:CRP}
\Pr(\C=C) = \frac{\alpha^{|C|}\Gamma(\alpha)}{\Gamma(\alpha + n)} \prod_{c\in C} \Gamma(|c|).
\end{align}
The sequential sampling process corresponding to this partition distribution is known as the Chinese restaurant process (CRP), or Blackwell--MacQueen urn process. 

The following key fact is a direct consequence of these two results (Sethuraman's and Antoniak's): if $\bm\pi=(\bm\pi_1,\bm\pi_2,\ldots)$ is defined as above, $\z_1,\ldots,\z_n|\bm\pi\,\iid\, \bm\pi$, and $\C$ is the partition induced by $\z_1,\ldots,\z_n$, then the distribution of $\C$ is given by Equation \ref{equation:CRP}.
This can be seen by noting that when $H$ is nonatomic, the distribution of $\C$ is the same as when it is induced by $\x_1,\ldots,\x_n|\bm P$.  

While this key fact follows directly from the results of Sethuraman and Antoniak, the proofs of their results are rather abstract
and are not easy to verify, especially for those without expertise in measure theory. 
The purpose of this note is to provide a proof of this connection between the CRP and the stick-breaking representation
using only elementary, non-measure-theoretic arguments.
Our proof is completely self-contained and does not rely on any properties of the Dirichlet process or other theoretical results.
Conversely, we also provide a sketch of how the CRP naturally leads to the stick-breaking representation.

In previous work, \citet{broderick2013cluster} used De Finetti's theorem to provide an elegant derivation of the stick-breaking weights from the CRP.  Also, \citet{Paisley_2010b} showed by elementary calculations that if the base distribution $H$ is a discrete distribution on $\{1,\ldots,K\}$, and $\bm P$ is defined by the stick-breaking process as in Equation~\ref{equation:stick}, then $(\bm P(1),\ldots,\bm P(K)) \sim \mathrm{Dirichlet}(\alpha H(1),\ldots,\alpha H(K))$; thus, despite the similar sounding title of the article by \citet{Paisley_2010b}, the result shown there is altogether different from what we show here.

\section{Main result}
\label{section:main}

We use $[n]$ to denote the set $\{1,\ldots,n\}$, and $\N$ to denote $\{1,2,3,\ldots\}$.
As is standard, we represent a partition of $[n]$ as a set $C=\{c_1,\ldots,c_t\}$ of nonempty disjoint sets $c_1,\ldots,c_t$
such that $\bigcup_{i=1}^t c_i = [n]$. Thus, $t=|C|$ is the number of parts in the partition, and $|c|$ is the number of elements in a given
part $c\in C$.
We say that $C$ is the partition of $[n]$ induced by $z_1,\ldots,z_n$ if it has the property that
for any $i,j\in[n]$, $i$ and $j$ belong to the same part $c\in C$ if and only if $z_i = z_j$.
We use bold font to denote random variables.

\begin{theorem}
    \label{theorem:main}
    Suppose
    \begin{align*}
        & \v_1,\v_2,\ldots\iid\Beta(1,\alpha) \\
        & \bm\pi_k = \v_k\prod_{i = 1}^{k-1} (1 -\v_i) \text{ for } k = 1,2,\ldots, \\
        & \z_1,\ldots,\z_n|\bm\pi=\pi \,\iid\, \pi, \text{ that is, } \Pr(\z_i = k\mid \pi) = \pi_k,
    \end{align*}
    and $\C$ is the partition of $[n]$ induced by $\z_1,\ldots,\z_n$. Then
    $$ \Pr(\C=C) = \frac{\alpha^{|C|}\Gamma(\alpha)}{\Gamma(n + \alpha)} \prod_{c\in C} \Gamma(|c|). $$
\end{theorem}

Our proof of the theorem relies on the following lemmas.
Let us abbreviate $z=(z_1,\ldots,z_n)$. 
Given $z\in\N^n$, let $C_z$ denote the partition $[n]$ induced by $z$.

\begin{lemma}\label{lemma:A}
  For any $z\in\N^n$,
  $$\Pr(\z=z) = \frac{\Gamma(\alpha)}{\Gamma(n+\alpha)}\Big(\prod_{c\in C_z}\Gamma(|c|+1)\Big)\Big(\prod_{k=1}^m \frac{\alpha}{g_k+\alpha}\Big)$$
  where $m = \max\{z_1,\ldots,z_n\}$ and $g_k = \#\{i:z_i\geq k\}$.
\end{lemma}

The proofs of the lemmas will be given in Section \ref{section:proofs}.
We use $\I(\cdot)$ to denote the indicator function, that is, $\I(E) = 1$ if $E$ is true, and $\I(E)=0$ otherwise.

\begin{lemma}\label{lemma:B}
    For any partition $C$ of $[n]$,
    $$\frac{\alpha^{|C|}}{\prod_{c\in C} |c|} = \sum_{z\in\N^n} \I(C_z=C) \prod_{k=1}^{m(z)} \frac{\alpha}{g_k(z)+\alpha}$$
    where $m(z) = \max\{z_1,\ldots,z_n\}$ and $g_k(z) = \#\{i:z_i\geq k\}$.
\end{lemma}

\begin{proof}[\bf Proof of Theorem \ref{theorem:main}]
\begin{align*}
\Pr(\C = C) & =\sum_{z\in\N^n}\Pr(\C = C\mid \z = z)\Pr(\z = z)\\
& \overset{\text{(a)}}{=}\sum_{z\in\N^n}\I(C_z = C)\frac{\Gamma(\alpha)}{\Gamma(n +\alpha)}
  \Big(\prod_{c\in C_z}\Gamma(|c|+1)\Big)\Big(\prod_{k = 1}^{m(z)}\frac{\alpha}{g_k(z) +\alpha}\Big)\\
& =\frac{\Gamma(\alpha)}{\Gamma(n +\alpha)}\Big(\prod_{c\in C}\Gamma(|c|+1)\Big)
  \sum_{z\in\N^n}\I(C_z = C)\Big(\prod_{k = 1}^{m(z)}\frac{\alpha}{g_k(z) +\alpha}\Big)\\
& \overset{\text{(b)}}{=}\frac{\Gamma(\alpha)}{\Gamma(n +\alpha)}\Big(\prod_{c\in C}\Gamma(|c|+1)\Big)\frac{\alpha^{|C|}}{\prod_{c\in C} |c|}\\
& \overset{\text{(c)}}{=}\frac{\Gamma(\alpha)}{\Gamma(n +\alpha)}\Big(\prod_{c\in C}\Gamma(|c|)\Big)\alpha^{|C|}
\end{align*}
where (a) is by Lemma \ref{lemma:A}, (b) is by Lemma \ref{lemma:B}, and (c) is since $\Gamma(|c|+1) =|c|\Gamma(|c|)$.
\end{proof}

\section{Proofs of lemmas}
\label{section:proofs}

\begin{proof}[\bf Proof of Lemma \ref{lemma:A}]
Letting $e_k =\#\{i:z_i = k\}$, we have
$$\Pr(\z = z\mid\pi_1,\ldots,\pi_m) =\prod_{i = 1}^n\pi_{z_i} =\prod_{k = 1}^m \pi_k^{e_k} $$
and thus
$$\Pr(\z = z\mid v_1,\ldots,v_m) =\prod_{k = 1}^m\Big(v_k\textstyle\prod_{i = 1}^{k -1} (1 - v_i)\Big)^{e_k}
=\displaystyle\prod_{k = 1}^m v_k^{e_k} (1 - v_k)^{f_k} $$
where $f_k =\#\{i:z_i>k\}$. Therefore,
\begin{align*}
\Pr(\z = z) & =\int \Pr(\z = z\mid v_1,\ldots,v_m) p(v_1,\ldots,v_m) d v_1\cdots d v_m\\
& =\int \Big(\prod_{k = 1}^m v_k^{e_k} (1 - v_k)^{f_k}\Big) p(v_1)\cdots p(v_m) d v_1\cdots d v_m\\
& =\prod_{k = 1}^m \int v_k^{e_k} (1 - v_k)^{f_k} p(v_k) d v_k\\
& \overset{\text{(a)}}{=}\prod_{k = 1}^m \alpha B(e_k +1,\, f_k +\alpha)\\
& =\prod_{k = 1}^m \frac{\alpha\Gamma(e_k +1)\Gamma(f_k +\alpha)}{\Gamma(e_k + f_k +\alpha+1)}\\
& \overset{\text{(b)}}{=}\prod_{k = 1}^m \frac{\alpha\Gamma(e_k +1)\Gamma(g_{k+1} +\alpha)}{\Gamma(g_k +\alpha+1)}\\
& \overset{\text{(c)}}{=}\Big(\prod_{k = 1}^m \Gamma(e_k +1)\Big)\Big(\prod_{k = 1}^m\frac{\alpha}{g_k +\alpha}\Big)
\Big(\prod_{k = 1}^m\frac{\Gamma(g_{k +1} +\alpha)}{\Gamma(g_k +\alpha)}\Big)\\
& =\Big(\prod_{c\in C_z}\Gamma(|c|+1)\Big)\Big(\prod_{k = 1}^m\frac{\alpha}{g_k +\alpha}\Big)\frac{\Gamma(\alpha)}{\Gamma (n +\alpha)}
\end{align*}
where step (a) holds since
$$\int x^r (1 - x)^s \Beta(x|1,\alpha) dx = \frac{B(r +1,\,s+\alpha)}{B(1,\alpha)} = \alpha B(r+1,\,s+\alpha),$$
step (b) since $f_k = g_{k +1}$ and $g_k = e_k + f_k$, and step (c) since $\Gamma(x +1) = x\Gamma(x)$.
\end{proof}

Let $S_t$ denote the set of $t!$ permutations of $[t]$.

\begin{lemma}\label{lemma:C}
For any $n_1,\ldots,n_t\in\N$,
$$\sum_{\sigma\in S_t} \frac{1}{a_1(\sigma)\cdots a_t(\sigma)} = \frac{1}{n_1\cdots n_t} $$
where $a_i(\sigma) = n_{\sigma_i} + n_{\sigma_{i +1}} +\cdots + n_{\sigma_t}$.
\end{lemma}
\begin{proof}
Consider an urn containing $t$ balls of various sizes---specifically, suppose the balls are labeled $1,\ldots,t$ and have sizes $n_1,\ldots,n_t$. Consider the process of sampling without replacement $t$ times from the urn, supposing that the probability of drawing any given ball is proportional to its size. This defines a distribution on permutations $\sigma\in S_t$ such that, letting $n =\sum_{i = 1}^t n_i$,
\begin{align*}
& p(\sigma_1) =\frac{n_{\sigma_1}}{n}=\frac{n_{\sigma_1}}{a_1(\sigma)},\\
& p(\sigma_2|\sigma_1)  =\frac{n_{\sigma_2}}{n - n_{\sigma_1}}=\frac{n_{\sigma_2}}{a_2(\sigma)},\\
& p(\sigma_3|\sigma_1,\sigma_2) =\frac{n_{\sigma_3}}{n - n_{\sigma_1} - n_{\sigma_2}}=\frac{n_{\sigma_3}}{a_3(\sigma)},
\end{align*}
and so on. Therefore, since $n_{\sigma_1}\cdots n_{\sigma_t} = n_1\cdots n_t$,
\begin{align}\label{equation:sigma}
p(\sigma) = p(\sigma_1) p(\sigma_2|\sigma_1)\cdots p(\sigma_t|\sigma_1,\ldots,\sigma_{t -1})
=\frac{n_1\cdots n_t}{a_1(\sigma)\cdots a_t(\sigma)}.
\end{align}
Since $p(\sigma)$ is a distribution on $S_t$ by construction, we have $\sum_{\sigma\in S_t} p(\sigma) = 1$; applying this to Equation \ref{equation:sigma} and dividing both sides by $n_1\cdots n_t$ gives the result.
\end{proof}

\begin{proof}[\bf Proof of Lemma \ref{lemma:B}]
Let $t =|C|$, and suppose $c_1,\ldots,c_t$ are the parts of $C$. For $\sigma\in S_t$, define $a_i(\sigma)=|c_{\sigma_i}|+\cdots +|c_{\sigma_t}|$. For any $z\in\N^n$ such that $C_z = C$, if $k_1<\cdots<k_t$ are the distinct values taken on by $z_1,\ldots,z_n$, then
\begin{align*}
\prod_{k = 1}^{m(z)}\frac{\alpha}{g_k(z) +\alpha} & =\Big(\frac{\alpha}{g_{k_1}(z) +\alpha}\Big)^{k_1}
\Big(\frac{\alpha}{g_{k_2}(z) +\alpha}\Big)^{k_2 - k_1} \cdots\Big(\frac{\alpha}{g_{k_t}(z) +\alpha}\Big)^{k_t - k_{t -1}} \\
& =\Big(\frac{\alpha}{a_1(\sigma) +\alpha}\Big)^{d_1}
\Big(\frac{\alpha}{a_2(\sigma) +\alpha}\Big)^{d_2} \cdots\Big(\frac{\alpha}{a_t(\sigma) +\alpha}\Big)^{d_t}
\end{align*}
where $d_i = k_i - k_{i -1}$, with $k_0 = 0$, and $\sigma$ is the permutation of $[t]$ such that $c_{\sigma_i} =\{j:z_j = k_i\}$. Note that the definition of $d=(d_1,\ldots,d_t)$ and $\sigma$ sets up a one-to-one correspondence (that is, a bijection) between $\{z\in\N^n:C_z = C\}$ and $\{(\sigma,d): \sigma\in S_t, \, d\in\N^t\}$. Therefore,
\begin{align*}
\sum_{z\in\N^n}\I(C_z = C)\prod_{k = 1}^{m(z)}\frac{\alpha}{g_k(z) +\alpha}
& =\sum_{\sigma\in S_t}\sum_{d\in\N^t}\prod_{i = 1}^t\Big(\frac{\alpha}{a_i(\sigma) +\alpha}\Big)^{d_i}\\
& =\sum_{\sigma\in S_t}\prod_{i = 1}^t\sum_{d_i\in\N}\Big(\frac{\alpha}{a_i(\sigma) +\alpha}\Big)^{d_i}\\
& \overset{\text{(a)}}{=}\sum_{\sigma\in S_t}\prod_{i = 1}^t \frac{\alpha}{a_i(\sigma)}\\
& \overset{\text{(b)}}{=}\frac{\alpha^t}{\prod_{i = 1}^t|c_i|} =\frac{\alpha^t}{\prod_{c\in C}|c|}
\end{align*}
where step (a) follows from the geometric series, $\sum_{k = 1}^\infty x^k = 1/(1-x) - 1$ for $x\in[0,1)$, and step (b) is by Lemma \ref{lemma:C}.
\end{proof}

\section{Deriving the stick-breaking process from the CRP}
\label{section:stick-from-CRP}

We have provided an elementary derivation of the CRP from the stick-breaking process.
What about going the other direction?  Starting from the CRP, how might one arrive at the stick-breaking representation?
Here, we sketch out how the stick-breaking process arises naturally from the CRP.
This section should be viewed as a concise exposition of existing results; see \cite{Pitman_2006} for reference.
In this section only, we appeal to Kingman's paintbox representation of exchangeable partitions, but otherwise our treatment is self-contained.

The CRP is a sequential allocation of customers $i=1,2,\ldots$ to tables $k=1,2,\ldots$ in which 
customer $1$ sits at table $1$, and each successive customer sits at a currently occupied table with probability proportional to the number of
customers at that table, or sits at the next unoccupied table with probability proportional to $\alpha$.
The resulting random partition of customers by table is distributed according to Equation~\ref{equation:CRP}.

First, consider table $1$.
Let $\y_i=1$ if customer $i$ sits at table $1$, and $\y_i=0$ otherwise.  
Then $\y_1=1$, and $\y_2,\y_3,\ldots$ is a two-color P\'{o}lya urn process in which $\y_i\mid \y_1,\ldots,\y_{i-1} \sim \mathrm{Bernoulli}\big(\sum_{j=1}^{i-1} \y_j / (\alpha + i-1)\big)$. Thus, 
\begin{align*}
p(y_1,\ldots,y_n) &= 
\frac{1 \cdot 2 \cdots (s_n - 1)\, \alpha (\alpha+1)\cdots (\alpha+n - s_n - 1)}{(\alpha+1)(\alpha+2)\cdots(\alpha+n-1)} \\
&= \frac{B(s_n,\, \alpha+n-s_n)}{B(1,\alpha)} = \int v^{s_n-1} (1-v)^{n-s_n} \mathrm{Beta}(v \mid 1,\alpha) d v,
\end{align*}
where $s_n = \sum_{i=1}^n y_i$.
Therefore, the same distribution on $\y_1,\y_2,\ldots$ can be generated by drawing $\v\sim\mathrm{Beta}(1,\alpha)$, then
setting $\y_1=1$ and drawing $\y_2,\y_3,\ldots|\v\iid \mathrm{Bernoulli}(\v)$. 
Letting $\v_1 = \lim_{n\to\infty} \frac{1}{n}\sum_{i=1}^n \y_i$ (the limit exists with probability 1), it follows that $\v_1 \sim \mathrm{Beta}(1,\alpha)$.
Note that $\v_1$ is the asymptotic proportion of customers at table $1$.

Now, consider table $k$.
Given the indices of the subsequence of customers that do not sit at tables $1,\ldots,k-1$,
the customers in this subsequence sit at table $k$ according to the same urn process as $\y_1,\y_2,\ldots$ above,
independently of the corresponding urn processes for $1,\ldots,k-1$.
Thus, of the customers not at tables $1,\ldots,k-1$, the proportion at table $k$ converges to a $\mathrm{Beta}(1,\alpha)$ random variable, say $\v_k$, independent of $\v_1,\ldots,\v_{k-1}$.
Therefore, out of all customers, the proportion at table $k$ converges to $\bm\pi_k := \v_k\prod_{j=1}^{k-1} (1 - \v_j)$ with probability 1.
(The preceding urn-based derivation is adapted from \citealp{broderick2013cluster}.)

Although we have arrived at the stick-breaking process, our derivation is not yet complete because the 
partition distribution given $\bm\pi$ as defined above is different than the partition distribution induced by assignments $\z_1,\ldots,\z_n|\bm\pi\iid \bm\pi$.
To establish that they are equivalent, marginally, we use Kingman's paintbox representation for exchangeable partitions.

By \cite{kingman1978representation}, 
there exists a random sequence $\q = (\q_1,\q_2,\ldots)$ with $\q_1\geq \q_2\geq \cdots \geq 0$ and $\sum_{j=1}^\infty \q_j = 1$ (with probability 1) such that 
the random partition induced by $\z_1,\ldots,\z_n|\q\iid \q$ is marginally distributed according to Equation~\ref{equation:CRP};
for a concise proof, see \cite{aldous1985exchangeability}, Prop.\ 11.9.
(Note that in general, $\sum_{j=1}^\infty \q_j < 1$ is possible, but not in this case because in the CRP, singleton clusters have probability 0, asymptotically.)
Let $\z_1,\z_2,\ldots|\q\iid \q$ and $\hat{\q}_{n j} = \frac{1}{n} \sum_{i=1}^n \I(\z_i = j)$.
With probability 1, for all $j\in\N$, $\hat{\q}_{n j} \to \q_j$ as $n\to\infty$,
by the law of large numbers.
Let $\bm\sigma = (\bm\sigma_1,\bm\sigma_2,\ldots)$ be the permutation of $\N$ such that $\bm\sigma_k$ is the $k$th distinct value to appear in $\z_1,\z_2,\ldots$; for reference, $\bm\sigma$ is called a size-biased permutation.
In the CRP terminology, customer $i$ is at table $k$ when $\z_i = \bm\sigma_k$, so the asymptotic proportion of customers at table $k$ is $\q_{\bm\sigma_k}$.
Therefore, by the urn derivation above, 
$(\q_{\bm\sigma_1},\q_{\bm\sigma_2},\ldots)$ is equal in distribution to $(\bm\pi_1,\bm\pi_2,\ldots)$ where $\bm\pi_k = \v_k\prod_{j=1}^{k-1} (1 - \v_j)$ and $\v_1,\v_2,\ldots\iid\mathrm{Beta}(1,\alpha)$.
Finally, note that permuting the entries of $\q$ (even via a random permutation that depends on $\q$) does not affect the marginal distribution
of the partition induced by $\z_1,\ldots,\z_n|\q\iid \q$.
This shows that the partition induced by $\z_1,\ldots,\z_n|\bm\pi\iid \bm\pi$ is marginally distributed according to Equation~\ref{equation:CRP}.

\section*{Acknowledgments}

Thanks to David Dunson, Garritt Page, Tamara Broderick, and Steve MacEachern for helpful conversations.

\bibliographystyle{abbrvnatcap}
\bibliography{refs}

\end{document}